\documentclass[a4paper,11pt]{amsart}
\usepackage{hyperref}
\usepackage[utf8]{inputenc}
\usepackage[all]{xy}
\usepackage{graphicx}
\usepackage{amssymb}
\usepackage{amsmath}
\usepackage{bbm}
\usepackage{leftidx}
\usepackage{yfonts}
\usepackage{mathtools}
\newtheorem{theorem}{Theorem}
\newtheorem{proposition}{Proposition}
\newtheorem{lemma}{Lemma}
\newtheorem{corollary}{Corollary}
\newtheorem{remark}{Remark}
\newtheorem*{remark*}{Remark}
\newtheorem{definition}{Definition}
\newtheorem*{theorem*}{Theorem}
\newtheorem*{conjecture*}{Conjecture}

\newtheorem*{notation*}{Notation}

\newtheorem*{app*}{Application to analytic GIT-quotients}
\newtheorem*{appl*}{Application}

\numberwithin{equation}{section}

\newcommand{\C}{{\mathbb C} }

\newcommand{\cA}{{\mathcal A} }
\newcommand{\cB}{{\mathcal B} }

\newcommand{\cE}{{\mathcal E} }
\newcommand{\cF}{{\mathcal F} }
\newcommand{\cG}{{\mathcal G} }

\newcommand{\cK}{{\mathcal K} }

\newcommand{\cM}{{\mathcal M} }
\newcommand{\cN}{{\mathcal N} }
\newcommand{\cO}{{\mathcal O} }
\newcommand{\cP}{{\mathcal P} }
\newcommand{\cQ}{{\mathcal Q} }

\newcommand{\cX}{{\mathcal X} }

\newcommand{\cH}{{\mathcal H} }

\newcommand{\wt}{\widetilde}
\newcommand{\wh}{\widehat}

\newcommand{\tbullet}{{\mbox{{\tiny\textbullet}}}}
\newcommand{\xdownarrow}[1]{%
  {\left\downarrow\vbox to #1{}\right.\kern-\nulldelimiterspace}
}
\makeatletter
\newcommand*\bigcdot{\mathpalette\bigcdot@{.5}}
\newcommand*\bigcdot@[2]{\mathbin{\vcenter{\hbox{\scalebox{#2}{$\m@th#1\bullet$}}}}}
\makeatother

\def\ol#1{{\overline{#1}}}
\def\ul#1{{\underline{#1}}}

\def\pt{\partial}

\def\we{\wedge}

\def\db{\ol\pt}
\def\s2{/\hspace{-3pt}/}
\def\san2{/\hspace{-3pt}/_{\hspace{-1pt}an}}

\makeatletter
\def\blfootnote{\xdef\@thefnmark{}\@footnotetext}
\makeatother

\def\ka{K\"ah\-ler}

\def\jh{Seshadri}

\author{Nicholas Buchdahl}
\address{School of Mathematical Sciences\\ University of Adelaide\\ Adelaide\\
Australia 5005}
\email{nicholas.buchdahl@adelaide.edu.au}

\author{Georg Schumacher}
\address{Fachbereich Mathematik und Informatik, Philipps-Universit\"at Marburg\\ Lahnberge, Hans-Meerwein-Stra{\ss}e, D-35032 Marburg, Germany}
\email{schumac@mathematik.uni-marburg.de}

\begin{document}

\title[Analytic Application of Geometric Invariant Theory II]{An Analytic Application of \\ Geometric Invariant Theory II:\\ Coarse Moduli Spaces}

\begin{abstract}
In \cite{b-s} the authors constructed a classifying space for polystable holomorphic vector bundles on a compact \ka\ manifold using analytic GIT theory. The aim of this article is to show that this classifying space taken in the weakly normal category is a coarse moduli space in the sense of complex geometry when the topology is fixed as induced by the space of Hermite-Einstein connections modulo the group of unitary gauge transformations.
\end{abstract}



\maketitle

\section{Introduction}

Let $X$ be a compact \ka\ manifold. In \cite{b-s} a reduced complex space $\cM_{GIT}$ was constructed, whose points correspond to isomorphism classes of {\em polystable} holomorphic vector bundles on $X$ by means of analytic Geometric Invariant Theory. It contains the coarse moduli space $\cM$ of {\em stable} holomorphic vector bundles as the complement of a closed analytic subset. This space is a union of local analytic GIT-spaces $S\san2 G$, where $S$ is the parameter space of a semi-universal deformation of a polystable bundle $E_0$, and $G$ is the complex reductive group of automorphisms of $E_0$ (modulo ineffectivity kernel). Such a space was called {\em classifying space}, because it satisfies slightly weaker axioms than a coarse moduli space.

In this article, the question of the existence of a coarse moduli space of polystable holomorphic vector bundles is addressed. The main problem is that in general there does not exist a deformation theory for (the category of) polystable vector bundles, where all fibers must be polystable. The notion of a coarse moduli space is closely related to the notion of holomorphic families of such objects. Since any semi-universal (local) deformation of a polystable vector bundle taken in the class of holomorphic vector bundles contains semi-stable but not polystable bundles in general, a GIT construction was necessary which identifies semistable bundles with polystable ones, if these are in the closure of the same $G$-orbit. For these reasons the construction of a coarse moduli space is fundamentally different from any construction that is mainly based upon deformation theory, where the uniqueness of a coarse moduli space always follows automatically from the existence of universal deformations for the given class of objects.

From the standpoint of complex analysis a coarse moduli space would be equipped with the quotient topology induced by holomorphic families of polystable bundles.

On the other hand, from the outset the space $\cM_{GIT}$ is equipped with a natural (Hausdorff) topology, namely the space of integrable, Hermite-Einstein connections modulo unitary gauge transformations (see the work of Donaldson-Kronheimer \cite[Sect.~4.2.1]{d-k} or also of Atiyah, Hitchin, and Singer \cite[\S6]{AHS}). An analogous situation for cscK manifolds has been treated by Dervan and Naumann in \cite{d-n}. In this sense the topology of a coarse moduli space can be considered as given (the classical moduli space of stable bundles carries this topology \cite{b-s}).

The existence of a coarse moduli space (equipped with a given topology) of polystable holomorphic vector bundles will be shown in the category of weakly normal complex spaces.

It was shown in \cite{b-s} that the closure of the moduli space of stable bundles in $\cM_{GIT}$ satisfies the axioms of a coarse moduli space. Here the technical method is to treat {\em \jh} filtrations which are also known as {\em Jor\-dan-Höl\-der} filtrations of length larger than one.

For any irreducible component of a parameter space $S$ the length of a \jh\ filtration need not be constant but it is minimal over the complement of a thin subset. The term thin is used here for an analytically constructible set that is nowhere dense (see Lemma~\ref{le:locan} below). Ultimately for any irreducible component of $S$ an analytic subset can be constructed with the same image in the analytic GIT-quotient such that the complement of a thin subset parameterizes polystable bundles in the sense of a holomorphic family. In this way the non-existence of a deformation theory for polystable bundles is circumvented by constructing a theory of families of filtrations.

\section{Weakly Normal Complex Spaces}\label{se:wenorm}
Some facts need to be gathered (cf.\ \cite{fi}).
A reduced complex space $(S,\cO_S)$ is called {\em weakly normal} or {\em maximal}, if all continuous functions that are holomorphic on the complement of the set of singularities are actually holomorphic. There exists a holomorphic map $\nu: \wh X \to X$, where $\wh X$ is weakly normal and $\nu$ is a homeomorphism which is biholomorphic over the complement of the set of non-maximal/non-weakly normal points in $X$. The set of weakly normal complex spaces together with holomorphic mappings is known to constitute a category. According to Andreotti and Norguet \cite{an} the weak normalization $\wh\pi:\wh X \to X$ is the quotient of the usual normalization $X^{n}/\sim$ by the proper analytic equivalence relation $\sim$, where two points are identified if their images in $X$ coincide.

\begin{proposition}\label{pr:normM}
Let $\cM_{GIT}$ be the classifying space for polystable vector bundles on the compact \ka\ manifold $(X,\omega_X)$. Then the weak normalization $\wh\cM_{GIT}$ is a classifying space for polystable vector bundles in the category of weakly normal complex spaces: i.e.\ Definition~4 of \cite{b-s} holds for weakly normal parameter spaces.
\end{proposition}

The following fact for restricted group actions in the sense of \cite[Sec.~2.1]{b-s} is mentioned first. Let a complex reductive group $G$ act linearly on a vector space $V$, and let $S$ be a closed analytic subset of an open subset $U\subset V$ amounting to a restricted action on $S$. The restricted action of $G$ is extended to a holomorphic action on a larger space $\wt S$ containing $S$ defined in \cite[Sec.~2.1]{b-s}.

\begin{lemma}\label{le:normS}
 The weak normalization $\wh{\wt S\san2 G}$ of the analytic quotient is equal to the analytic quotient $\wh{\wt S}\san2 G$ of the  weak normalization.
\end{lemma}
\begin{proof}
  The action of $G$ on $\wt S$ clearly extends to the weak normalization $\wh{\wt S}$. In this way the analytic GIT-quotient $\wt{\wh S}\san2 G$ is defined along the lines of \cite[Sec.\ 2]{b-s}. As $\wh S$ is weakly normal, so are $\wt{\wh S}$, and finally $\wt{\wh S}\san2 G$. The latter space is equal to $\wh{\wt S\san2 G}$ by \cite[Theorem 3]{b-s}.
\end{proof}

\begin{proof}[Proof of the Proposition]
The notation from \cite{b-s} is used. First the space $\wh \cM_{GIT}$ is equipped with the local structure of an analytic GIT-quotient.

Let $S$ be the parameter  space for a semi-universal deformation of a polystable holomorphic vector bundle $E_0$ on $X$. Then the weak normalization $\wh S \to S$ defines a pull-back of the semi-universal deformation. It follows immediately that the group $G$ of automorphisms of $E_0$ acts on $\wh S$ by pulling back the action on $S$ so that Lemma~\ref{le:normS} is applicable. Since the weak normalization map $\wh Y \to Y$ of a reduced complex space $Y$ is a homeomorphism, the underlying points of $\wh \cM_{GIT}$ still correspond to isomorphism classes of polystable bundles.  This implies condition (i) of \cite[Definition~4]{b-s}.

Now let a family of holomorphic vector bundles over a weakly normal space $Z$ be given such that the fiber of a distinguished point is polystable (cf.\ \cite[Definition~4]{b-s}). This family is induced (after replacing the parameter space by a neighborhood of the distinguished point, if necessary) by a semi-universal deformation via a base change map $Z \to S$, which can be lifted to a map $Z\to \wh S$. The latter map is composed with the canonical map $\wh S \to \wt{\wh S}\san2 G \subset \wh\cM_{GIT}$ implying condition (ii) of \cite[Definition~4]{b-s}.

Condition (iii) is the uniqueness of  $\wh \cM_{GIT}$ given (i) and (ii). Let $\cQ$ be another space satisfying conditions (i) and (ii). Let $Z$ be a weakly normal parameter space in the sense of (iii) with an induced  map $Z \to \cQ$. Now $Z$ is interpreted as a reduced parameter space so that the proof of \cite[Theorem~9]{b-s} can be applied. It yields a holomorphic map $Z \to \cM_{GIT}$ compatible with the set theoretic bijection $\cM_{GIT} \to \cQ$. Since $Z$ is weakly normal the above map can be lifted to a map $Z\to \wh{\cM_{GIT}}$ which proves (iii).
\end{proof}

The proposition below is already contained in \cite{b-s}. By \cite[Remark~5]{b-s} the set of polystable points in $\cM_{GIT}$ that are not stable is a closed analytic subset. In fact the complement is the usual coarse moduli space $\mathcal M$ of isomorphism classes of stable holomorphic vector bundles. Let $\wh \cM$ be the coarse moduli space of stable vector bundles taken in the category of weakly normal spaces.

For any local analytic GIT-quotient $\wt S\san2 G$ of a space $(S,s_0)$ consider those local irreducible components such that $s_0$ is in the closure of the set of points with {\em stable} fibers. These give rise to irreducible components of the local analytic GIT-quotients. The union is the closure $\ol\cM$ of $\cM$ in $\cM_{GIT}$. Since the underlying topology of a complex space is not changed by taking the weak normalization, the following argument can be applied. The normalization $\wh{\ol\cM}$ is equal to the union of the spaces $\wh{\wt S}\san2 G$ where the parameter spaces $S$ are taken with the above extra condition. In this way the proposition below follows.

\begin{proposition}\label{pr:Mbar}
  Let $\ol{\wh\cM} \subset \wh{\cM}_{GIT}$ be the closure taken in the classifying space. Then $\ol{\wh\cM}$ can be identified with the weak normalization $\wh{\ol \cM}$. It carries a natural structure of an analytic GIT-space.
\end{proposition}

\section{Coarse Moduli Spaces}
It was already mentioned that the uniqueness axiom (iii) below for a coarse moduli space requires special attention when GIT methods are applied.
\begin{definition}\label{de:coarse}
  Let $\cK$ be the class of polystable holomorphic vector bundles on a compact \ka\ manifold $X$. A reduced complex space $\cN$ is called coarse moduli space with a distinguished topology for $\cK$, if the following hold:
  \begin{itemize}
    \item[(i)] The points of $\cN$ correspond uniquely to isomorphism classes of polystable bundles on $X$.
    \item[(ii)] Let $\cE$ be a holomorphic family of polystable vector bundles on $X$ parameterized by a complex space $Z$. Then the natural map $\varphi: Z \to \cN$ sending a point $z\in Z$ to the isomorphism class of the fiber $\cE_z$ is holomorphic.
    \item[(iii)] Let $\cQ$ be another reduced complex space satisfying (i) and (ii) so that the induced set-theoretical natural map $\chi:\cN \to \cQ$ is a homeomorphism. Then $\chi$ is holomorphic.
  \end{itemize}
\end{definition}
Observe that the above notion is also meaningful if all parameter spaces of holomorphic families and morphisms are taken from the category of weakly normal spaces.

Consider the space $\ol{\wh\cM}$ from Proposition~\ref{pr:Mbar}. A first result on the existence of a coarse moduli space is the following.

\begin{proposition}
  The space $\ol{\wh\cM}$ is a coarse moduli space, namely the coarse moduli space of polystable holomorphic vector bundles that possess stable holomorphic bundles as local deformations taken in the category of weakly normal spaces.
\end{proposition}
\begin{proof}
Set-theoretically the space $\ol{\wh\cM}$ is exactly the space of isomorphism classes stated in the Theorem~\ref{th:main}. Let $Z$ be a (weakly normal) parameter space for a family of {\em polystable} bundles, where all fibers possess local deformations with stable fibers. Then the map $Z\to \wh{\cM}_{GIT}$ sending a point to its isomorphism class is holomorphic by Proposition~\ref{pr:normM}, and it has values in $\ol{\wh\cM}$. This proves condition (ii). In order to prove (iii) let $\cQ$ be a (weakly normal) space satisfying (i) and (ii). The set-theoretic map $\chi:\ol{\wh\cM} \to \cQ$ that is compatible with identifying points and isomorphism classes of bundles is a homeomorphism by assumption. It is known to be holomorphic when restricted to the coarse moduli space of stable bundles. The map $\chi$ restricted to the stable locus is holomorphic, hence $\chi$ is holomorphic by the Riemann removable singularity theorem and weak normality.
\end{proof}

\section{Summary of Previous Results}\label{se:prl}
The notation from \cite{bu-sch} is being recalled and some of the results will be listed. Let $d_0=\pt_0+\ol\pt_0$ be a hermitian connection on the hermitian bundle
$\rm E_h$ with curvature $F(d_0)$ being of type $(1,1)$
and satisfying $i\Lambda F(d_0))=\lambda\,1$ for some constant
$\lambda$, and let $E_0$ denote the polystable holomorphic bundle defined by
this connection. Denote by $A^{0,q}$ the global $(0,q)$-forms on $X$
and by $H^{0,q}$ the space of $\db_0$-harmonic $(0,q)$-forms
with coefficients in $\rm End(E_h)$.

With $p>2 \dim X$ fixed, from Lemma 2.5 and Proposition 2.6 of \cite{bu-sch}, there is an open $L^p_1$-neighbourhood of $d_0$ such that every connection in this neighbourhood is complex gauge equivalent to another one in this neighbourhood of the form $d_a=d_0+a$ where
\begin{eqnarray*}
a&=&a'+a''\\
a'&=&-(a'')^*\in A^{0,1}(\mathrm{End(E_h)})\\
a'' &=& \alpha+\db_0^*\beta
\end{eqnarray*}
for some $\alpha\in H^{0,1}$ and some unique $\beta\in A^{0,2}(\mathrm{End( E_h))}$ orthogonal in $L^2$ to $\ker \db_0^*$ such that
$$
\db_0^*(\db_0a''+a''\we a'')=0.
$$
The construction of $\beta$ is via the implicit function theorem, and $\beta$ is uniquely determined by $\alpha$. From Sec.~1 of \cite{bu-sch}, once $a\in A^1(\mathrm{End(E_h)})$ is in the ``standard'' form above, $||a||_{L^p_1} \le C||\alpha||_{L^2}$ for some constant $C=C(d_0)$.

There is an open ball $U$ centered at the origin in  $H^{0,1}$  and a holomorphic map
$$
H^{0,1}\supset U \stackrel{\Psi}{\longrightarrow} H^{0,2}
$$
defined by
\begin{equation}
\Psi(\alpha)= \Pi^{0,2}(\db_0a'' + a''\we a'')
\end{equation}
the zeros of which correspond precisely to the connections $d_a$ near $d_0$ in standard form that are integrable; i.e., have curvatures of type $(1,1)$. The integrability condition determines an analytic subset $\Psi^{-1}(0)=S\subset U$.
The notation $s\in S$ is natural when dealing with the parameter of a holomorphic family, which at the same time can be identified with a connection form $\alpha$ or $a''$ resp. So for $s\in S$ corresponding to an integrable semi-connection
\begin{equation}\label{eq:spfo}
\db_0 + a'' \text{ with } a''= \alpha+ \db^*_0\beta
\end{equation}
and $a= a' + a''$ given as above, the holomorphic bundle $E_\alpha$  is that determined by
$$
d_a = d_0 + a.
$$
The principal result of \cite[Theorem 3]{bu-sch} is that the holomorphic bundle $E_{\alpha}$ determined by $\alpha\in \Psi^{-1}(0)$ is [poly]stable if and only if
$\alpha$ is [poly]stable with respect to the action of the automorphism group $\Gamma$ of $E_0$ acting on $H^{0,1}$. The Kempf-Ness theorem identifies the
polystable orbits of $\Gamma$ as precisely those containing a zero of the moment map $m$ for the action, which is given in \cite[Cor.\ 5.3]{bu-sch}:
\begin{equation}\label{eq:mom}
m(\alpha) = \Pi^{0,0}i\Lambda(\alpha\wedge\alpha^*+
\alpha^*\wedge\alpha)\;.
\end{equation}
Throughout this article the following fact will be needed:
\begin{proposition}[{\cite[Prop.\ 3.2, 4.4]{bu-sch}}]\label{pr:destable}
 For $d_0$ as above, every integrable connection in an $L^p_1$ neighbourhood of $d_0$ defines a semi-stable holomorphic bundle $E$, and every destabilising subsheaf $A\subset E$ with torsion-free quotient is actually a subbundle.
\end{proposition}

According to \cite[Thm.~4.7]{bu-sch} with $d_a = d_0 + a$ in the standard form the following hold for any endomorphism $\sigma\in A^{0,0}(\mathrm{End(E_h)})$ such that $\ol\pt_a\sigma=0$:
\begin{eqnarray}
  d_0\sigma &=& 0 \label{eq:d0sig} \\
   {[}\alpha,\sigma ] &=& 0  \label{eq:alpcomm}\\
   {[}\beta,\sigma ] &=& 0. \label{eq:bcomm}
\end{eqnarray}
(In particular $\sigma$ is holomorphic with respect to the structure $\db_0$ defining $E_0$.) Furthermore, if $m(\alpha)=0$, then
\begin{equation}\label{eq:alpsig}
  [\alpha,\sigma^*]=0=[\beta,\sigma^*].
\end{equation}

Note that for the explicitly constructed family over $S$ all integrable semi-connections are of the special form \eqref{eq:spfo}. In general, when dealing with holomorphic families over arbitrary base spaces $W$ say, being induced by a base change morphism $W\to S$, the initial family of integrable semi-connections parameterized by $W$ is only complex-gauge equivalent to the pull-back of those in special form.

\section{General Case}\label{se:geca}
In the sequel \jh\ filtrations of holomorphic structures on $\rm E_h$ in the above sense are studied. It is important that in this case, namely for bundles close to $E_0$, \jh\ filtrations consist of subbundles rather than coherent sheaves (see {\cite[\S9]{bu-sch}} and Sec.~\ref{se:proofmain}). By definition, all occurring subbundles will have the same slope and quotients of successive terms are stable. Families of filtrations will play a central role.

\subsection{Deformations of polystable bundles}\label{su:defpol} Let again $E_0$ be a polystable bundle, and $E=E_\alpha$ be a polystable bundle close to $E_0$ in the sense of Sec.~\ref{se:prl}. Let $A\subset E$ be a holomorphic subbundle of the same slope.
Then $A$ is necessarily polystable (as can be seen by considering an Hermite-Einstein connection on $E$ and the connection it induces on $A$), and there is a splitting $E=A\oplus B$ for some other polystable subbundle $B$ of $E$.

Let $\sigma$ be the endomorphism of $E$ that is the identity on $A$ and zero on $B$. Then $\db_a\sigma=0$, so by \eqref{eq:d0sig}, \eqref{eq:alpcomm} and \eqref{eq:bcomm}, $d_0\sigma=0$ and $\sigma$ commutes with $\alpha$ and with $\beta$. Consequently,
the polystable bundle $E_0$ splits into a direct sum of polystable subbundles $E_0=A_0\oplus B_0$, where $A_0$ (resp.\ $B_0$) is the hermitian bundle underlying $A$
(resp.\ $B$)  equipped with the restriction of $d_0$. Of course, $A_0, B_0$ need not be stable even if $A$ or $B$ is stable.

If $m$ is the moment map for the action of $\Gamma= Aut(E_0)$ on $H^{0,1}$ \eqref{eq:mom} and $m(\alpha) = 0$, then more can be said.  In this case, from \eqref{eq:alpsig} it follows that $d_a\sigma=0$, and since $\sigma$ is injective on the image of $\sigma^*$, the orthogonal projection $\sigma(\sigma^*\sigma)^{-1}\sigma^*$
of $\rm E_h$ onto the hermitian bundle underlying $A$ is covariantly constant with respect to $d_a$. It is also covariantly constant with respect to $d_0$ and commutes with $\alpha$ and with $\beta$. Replacing $\sigma$ by this orthogonal (and holomorphic) projection yields the following proposition.

\begin{proposition}\label{pr:sigmanew}
Let $E=E_{\alpha}$ be a polystable bundle as above with $m(\alpha)=0$. Let $A\subset E$ be a holomorphic subbundle of the same slope. Then orthogonal projection onto $A$ followed by inclusion into $E$ is a self-adjoint endomorphism $\sigma$ of $\rm E_h$ satisfying $\sigma^2=\sigma$, $d_a\sigma=0$, $d_0\sigma=0$, $[\alpha,\sigma]=0$, and $[\beta,\sigma]=0$. Thus there are compatible holomorphic splittings $E=A\oplus B$ and $E_0=A_0\oplus B_0$ which are orthogonal with respect to the underlying hermitian
metric $h$ on $\rm E_h$.
\end{proposition}

\begin{corollary}\label{co:numsum}
The number of stable summands of $E_\alpha$ is less than or equal to the
number of stable summands of $E_0$.
\end{corollary}

Given $A\subset E$ as in Proposition~\ref{pr:sigmanew} the polystable subbundle $A_0\subset E_0$ is determined by $A$.  The bundle $E_0$ splits as a direct sum of stable summands, each of the same slope. The holomorphic structure of each summand is that induced by the reducible connection $d_0$ restricted to the corresponding $d_0$-invariant subbundle.
If $E_1,\dots, E_m$ are the distinct stable components of $E_0$, then $E_0=\bigoplus_{j=1}^mV_j\otimes E_j$ where $V_j$ is a hermitian vector space of dimension equal to the multiplicity of $E_j$ in $E_0$. Then $\Gamma := Aut(E_0) = \Pi_{j=1}^m GL(V_j)$ and $H^{0,1}=\bigoplus_{j,k=1}^m \big(Hom(V_j,V_k) \otimes H^{0,1}(Hom(E_j,E_k))\big)$.

If $A_0\subset E_0$ is a polystable subbundle with $\mu(A_0) = \mu(E_0)$, then each stable component of $A_0$ must be isomorphic to a stable component of $E_0$. Thus $A_0=\bigoplus_{j=1}^m W_j\otimes E_j$ where $W_j\subset V_j$ is a subspace (possibly $0$). For a hermitian vector space $V$ such as one of the spaces  $V_j$, the group $U(V)\subset GL(V)$ acts transitively on the Grassmannian $\mathrm{Grass}_k(V)$ of $k$-dimensional subspaces of $V$, and therefore the group $U(\Gamma)\subset\Gamma$ acts transitively on the set of polystable subbundles of $E_0$ isomorphic to $A_0$.
The conclusion is the following.
\begin{proposition}\label{pr:aut}
  Two isomorphic subbundles $A_0\subset E_0$ and $A'_0 \subset E_0$ of the same slope differ by a unitary automorphism of $E_0$. So the isomorphism classes of such subbundles are determined by the dimensions of the spaces $W_j$.
\end{proposition}

Now the construction from Section~\ref{se:prl} can be compared to the analogous
construction of a semi-universal deformation of $A_0$.

The harmonic spaces of $(p,q)$-forms on $X$ with values in a bundle
$\cF$ will be denoted by $H^{p,q}(\cF)$. The orthogonal decomposition
$E_0=A_0\oplus B_0$ yields an orthogonal decomposition of $H^{0,q}(End(E_0))$.
The above construction of a semi-universal deformation of $A_0$ given
in the previous section will be applied, using subscripts $A_0$.

Let $E=E_\alpha$ be polystable with $m(\alpha)=0$ and let  $A \subset E$ and $A_0\subset E_0$ be as in Proposition~\ref{pr:sigmanew}.

\begin{lemma}\label{le:Psi}
\begin{itemize}
  \item[(i)]
The following diagram commutes:
\begin{equation}\label{eq:diagram}
  \begin{array}{ccc}
   H^{0,1}(End(A_0)) & \stackrel{\iota}{\longrightarrow} & H^{0,1}(End(E_0))\\
   \cup && \cup\\
   U_{A_0}&\longrightarrow& U\\
   \strut\hskip8mm$\mbox{$\xdownarrow{5mm}$ ${\Psi_{A_0}}$}$ &&
   \strut\hskip6mm$\mbox{$\xdownarrow{5mm}$ ${\Psi}$}$\\
   H^{0,2}(End(A_0)) & \stackrel{}{\longrightarrow} & H^{0,2}(End(E_0))
  \end{array}
\end{equation}
\item[(ii)] The sets $U,U_{A_0}$ can be chosen so
that $\iota(U_{A_0})= U \cap \iota(H^{0,1}(End(A_0)))$ holds
and so that a
parameter space $S_{A_0}= \Psi^{-1}_{A_0}(0)$ for a semi-universal
deformation of $A_0$ satisfies $\iota(S_{A_0})=S\cap \iota(H^{0,1}(End(A_0)))$,
where $S=\Psi^{-1}(0)$ is a parameter space for a semi-universal
deformation of $E_0$.

\item[(iii)] Using the obvious notation,
$$
m^{-1}(0)\cap \iota (H^{0,1}(End(A_0)))=\iota(m^{-1}_{A_0}(0)).
$$
\end{itemize}
\end{lemma}
The {\em proof} follows from the orthogonal splitting of cohomology
groups, \eqref{eq:d0sig}, \eqref{eq:alpcomm}, and \eqref{eq:bcomm}
and the construction of the functions
$\Psi_{E_0},\Psi_{A_0}$, which are equivariant with respect to the
actions of the respective automorphism groups. \qed

Later the statement will be needed simultaneously for several
subbundles of $E_0$ corresponding to a filtration, namely the fact
that $U_{A_0}$ is determined by $U\subset H^{0,1}(End(E_0))$ and $H^{0,1}(End(A_0))$,  and not by the choice of $A_0$.

\subsection{Relative extension sheaves}\label{se:relext}
For the algebraic case refer to \cite[EGA III 7.7.8, 7.7.9]{gro}; the analytic case is based upon \cite{gra}. In the present context again all parameter spaces are reduced by assumption.
\begin{definition}\label{de:relex}
Let $f:\cX\to S$ be a proper, flat, holomorphic map of complex spaces, and $\cF$, $\cG$ be coherent sheaves on $\cX$. Then the presheaf
$$
S\supset W \mapsto Ext^q(f^{-1}(W); \cF, \cG)
$$
of equivalence classes of $q$-extensions of $\cF|f^{-1}(W)$ by $\cG|f^{-1}(W)$ determines the relative Ext-sheaves
$$
\cE xt^q(f; \cF, \cG)(W).
$$
Given a base change map $g:T\to S$ with $f':X_T=\cX\times_S T \to T$, and $g':\cX_T \to \cX$, the natural base change morphism is denoted by
\begin{equation}\label{eq:bch}
g^*\cE xt^q(f; \cF, \cG) \to \cE xt^q(f';\cF_T, \cG_T)
\end{equation}
with $\cF_T= g'^*\cF$ and $\cG_T= g'^*\cG$.
\end{definition}
In order to avoid unnecessary complications the sheaf $\cF$ is usually assumed to be $S$-flat.

A spectral sequence argument implies that sections of a relative ext-sheaf over a Stein space $W$ correspond to elements of the corresponding ext group over $f^{-1}(W)$.

\begin{remark}\label{re:grau}
The base change map is an isomorphism under extra assumptions (Grauert Base Change Theorem) which imply the freeness of the relative ext-sheaves.  However, in our applications these \ul{cannot} be made.
\end{remark}

By \cite[Satz~1]{bps} (assuming condition {\rm (i)} of the theorem, which will be satisfied in our applications) there exists locally with respect to $S$ a complex $\cP^\tbullet$ of finite free $\cO_S$-modules bounded to the right and isomorphisms
\begin{equation}\label{eq:rep}
  \cE xt^q(f'; \cF_T, \cG_T) \stackrel{\sim}{\longrightarrow} \cH^q(g^* \cP^\tbullet)
\end{equation}
that are compatible with the base change morphisms \eqref{eq:bch} and
\begin{equation}\label{eq:bch2}
g^*(\cH^q(\cP^\tbullet)) \to \cH^q(g^*\cP^\tbullet).
\end{equation}
Of particular interest are the base change morphisms $\{s\} \to S$ for points $s\in S$.

For any $q$ denote the support of the respective sheaf $\cH^q(\cP^\tbullet)$ by $S_q \subset S$, which is a closed analytic subset. The linear space over $S_q$ associated to $\cH^q(\cP^\tbullet)$ in the sense of \cite[Sec.~1.4]{fi} will be denoted by $V_q \to S_q \subset S$. (Here the space $S$ may have to be replaced by a neighborhood of a prespecified point.)

The above definition will be applied to the case where $\cX=X \times S$ and $f$ is the projection onto $S$. Let $\cF_s=\cF| X\times\{s\}$ and $\cG_s=\cG| X\times\{s\}$ for $s\in S$; then
\begin{equation}\label{eq:Exts}
  Ext^q(X;\cF_s,\cG_s) = V_{q,s}
\end{equation}
where $V_{q,s}$ denotes the fiber of $s$ of the linear space $V_q\to S$. Here the description of the extension groups in terms of the cohomology \eqref{eq:rep} of a bounded complex of free $\cO_S$-modules is crucial.

Altogether there is the following application for the cases $q=0,1$:
\begin{proposition}\label{pr:homext}
Let $S$ be Stein. Then the following hold:
\begin{itemize}
\item[(i)]
There exists a universal homomorphism
$$
\Phi: \cF_{V_0} \to \cG_{V_0}
$$
over $X\times V_0$ such that for all $s\in S_0\subset S$ the fiber of $V_0\to S_0$  consists of all homomorphisms $\cF|X\times\{s\}$ to $\cG|X\times\{s\}$.
\item[(ii)]
There exists a universal extension
$$
0 \to \cF_{V_1} \to \cE \to \cG_{V_1} \to 0
$$
over $X\times V_1$ such that for all $s\in S_1\subset S$ the fiber of $V_1\to S_1$ consists of all extensions of $\cF_s$ by $\cG_s$ up to equivalence.
\end{itemize}
\end{proposition}
The statement of Proposition~\ref{pr:aut} has another implication. The (restricted) action of the group $\Gamma=Aut(E_0)$  on the base space $S$ of a semi-universal deformation of $E_0$ in the sense of \cite{bu-sch} plays a central role in the construction of a coarse moduli space. Even though the choice of an isomorphism of $E_0$ and the distinguished fiber of a semi-universal deformation changes the latter, the following fact holds.

Let $A_0, A'_0\subset E_0$ be isomorphic subbundles. Then by Proposition~\ref{pr:aut} these subbundles differ by a unitary automorphism  $\delta$ of $E_0$ which will be kept fixed in the sequel. We recall the implication for the  action of $\Gamma$ on the germ $(S,0)$ of $S$ at $0$. This action proven in \cite[Sec.~2.1]{b-s} can be summarized in terms of a diagram:
\begin{equation*}
\scalebox{.75}{
  \xymatrix{
  & E_0 \ar[rr]\ar[dd]& &\cE\ar[dd] \\
  E_0 \ar[rr] \ar[ur]^\varphi \ar[dd] & & \cE \ar[ur]^\psi \ar[dd] &  \\
  & 0 \ar[rr]& & S  \\
  0 \ar[rr] \ar@{=}[ru] & &S \ar[ur]^\Phi &
  }
  }
\end{equation*}
Here the map $\psi$ is an isomorphism from $\cE$ to $\Phi^*\cE$, and the representation $\Gamma \to Aut(S,0)$ is given by $\varphi \mapsto \Phi$. The application of $\delta$ amounts to the conjugate action of $\varphi\mapsto    \delta^{-1}\circ\varphi\circ\delta \mapsto \Delta^{-1}\circ\Phi\circ\Delta$, where $\Delta$ is the image of $\delta$ under the former representation.

It is also useful to consider the action of $Aut(E_0)$ on extensions of the form $0\to A_0 \to E_0 \to B_0 \to 0$ (which are equivalent to the trivial extension), in particular in view of Proposition~\ref{pr:aut}, where an automorphism of $E_0$ is determined (non-uniquely) by an isomorphism of subbundles $A_0$ and $A'_0$.
\begin{lemma}\label{le:iso}
The choice of $\delta$ as above does not affect the  pointwise quotient nor the analytic GIT-quotient of $S$. It induces an equivalence of the extension $0\to A_0 \to E_0 \to B_0 \to 0$.
\end{lemma}
The second assertion still need proving. It follows from the diagram associated to the embeddings $A_0\hookrightarrow E_0$ and $A'_0 \hookrightarrow E_0$ into the polystable bundle $E_0$ and the induced map on orthogonal complements.
\begin{equation*}
  \xymatrix{
  0 \ar[r]& A_0 \ar[r] \ar[dr] \ar[d]_{\phi|A_0}^*[@]{\!\!\sim} & E_0 \ar[r]\ar[d]_{\phi}^*[@]{\!\!\sim}& B_0 \ar[r]\ar[d]_{\ul\phi}^*[@]{\!\!\sim} & 0 \\
  0  \ar[r] & A'_0 \ar[r]& E_0\ar[r]\ar[ur] & B'_0 \ar[r] & 0
  }
\end{equation*}
which yields an equivalence of extensions of $B_0$ by $A_0$.

\subsection{Choice of parameter spaces}\label{se:cho}
The deformation theoretic approach implies that parameter spaces are primarily defined as germs of complex spaces. In the course of constructions (including relative extension sheaves) these are represented by suitable spaces which possibly need to be replaced by smaller spaces.
This is achieved by shrinking the ball $U$; however, this need only be done a finite number of times by virtue of the fact that there are only finitely many isomorphism classes of subbundles of $E_0$ of the same slope.

The initial object is a polystable bundle $E_0$. The parameter space $S\subset U$ was constructed in \cite[Theorem 3]{b-s} (see Sec.~\ref{se:prl}). The open set $U\subset H^{0,1}$ was later specified as an open ball around zero. For all subbundles $A_0\subset E_0$ (with the same slope as $E_0$) semi-universal deformations exist in the sense of Lemma~\ref{le:Psi}. Here parameter spaces exist as subspaces of the fixed set $U$, which need not be replaced by a smaller neighborhood. The choice of subbundles $A_0\subset E_0$ only depends upon the isomorphism classes of $A_0$ by Proposition~\ref{pr:aut} and Lemma~\ref{le:iso} -- only finitely many such isomorphism classes exist. In a successive way universal spaces of extensions need to be considered. In \eqref{eq:rep} the sheaf of relative extensions is constructed locally as the cohomology of a bounded complex of free sheaves on the base. The support of the extension sheaf (giving rise to a linear space over the support) is a complex subspace. Here finitely many times the initial set $U\supset S$ may need shrinking. Also the case $q=0$ of homomorphisms is being treated in the analogous way (see Proposition~\ref{pr:homext} and Proposition~\ref{pr:polfib} below).  This process also guarantees that the supports of sheaves of the type $\cH^q(\cP^\tbullet)$, $q=0,1$ are closed analytic subspaces such that all irreducible components contain the fibers of $0\in S$. Including all isomorphism classes of subbundles and successive relative extension sheaves,  choices can be made before carrying out the actual construction, there being only a finite number of such choices required. These choices are controlled by specifying $U\subset H^{0,1}$, which determines parameter spaces for families of bundles and extensions.

\subsection{Families of filtrations}
The usefulness of filtrations for deformations of polystable bundles depends heavily upon the fact that destabilizing subsheaves are subbundles (Proposition~\ref{pr:destable}). The theory of deformations of filtrations of holomorphic vector bundles is part of the theory of deformations of multiple vector spaces equipped with linear maps over a compact complex space or manifold.

However, the ``size'' of parameter spaces of semi-universal deformations is of importance (see condition $(*)$ below). Because of this property the following arguments are necessary.

The results of the previous sections will be applied to filtrations. Again the polystable reference bundle is denoted by $E_0$ and the construction of a semi-universal deformation in the sense of Section~\ref{se:prl} is assumed with a parameter space $S\subset U$, in particular $s\in S$ denotes any polystable point with fiber $\cE_s=E=E_\alpha$. Now any filtration
\begin{equation}\label{eq:filtE}
  0 \subsetneq E^1 \subsetneq E^2 \subsetneq \cdots \subsetneq E^k  =E
\end{equation}
of length $k$ is taken, where the subbundles are of the same slope as $E$. In particular the filtration can assumed to be a \jh\ filtration, where the subsequent quotients are all stable bundles.

The process from Proposition~\ref{pr:sigmanew} is applied successively to $E^j$ and its subbundle $E^{j-1}$ beginning with $E=E^k$ yielding a filtration
\begin{equation}\label{eq:filtE_0}
0 \subsetneq E^1_0 \subsetneq E^2_0 \subsetneq \cdots\subsetneq E^k_0 \subsetneq E_0 .
\end{equation}
The aim is to construct a complete deformation of \eqref{eq:filtE_0} such that
\begin{itemize}
\item[$(*)$]
{\em the initially chosen filtration of $E=E_\alpha$ is contained in the respective family as a fiber.}
\end{itemize}
By definition a holomorphic family of filtrations is given by a family
$$
0 \subsetneq \cE^1 \subsetneq \cE^2 \subsetneq\cdots\subsetneq \cE^k
$$
of holomorphic vector bundles over $X\times R$ where $R$ denotes a parameter space $R$ together with an isomorphism of the fiber taken at the distinguished point $0\in R$ and \eqref{eq:filtE_0}.

The case $k=1$ is contained in Section~\ref{se:prl}. Filtrations of length $k=2$ are called ``short filtrations'' and will be treated in the next section.

\subsection{Short filtrations}\label{su:sfilt}
In order to simplify the notation, instead of using superscripts $k$ and $k-1$, the notation
$$
0\subsetneq A_0 \subsetneq  E_0
$$
will be used, given here
$$
0 \subsetneq A \subsetneq E
$$
with $m(\alpha)=0$. Assuming the situation of Proposition~\ref{pr:sigmanew} we have $E_0=A_0\oplus B_0$. Let $\cA$ and $\cB$ be the holomorphic families  of bundles over the parameter spaces $S_{A_0}$ resp.\ $S_{B_0}$ from Lemma~\ref{le:Psi}. Denote by $\cA_{S_{A_0}\times S_{B_0}}$ and $\cB_{S_{A_0}\times S_{B_0}}$ the families of bundles according to Lemma~\ref{le:Psi} pulled back to the common parameter space ${S_{A_0}\times S_{B_0}}$. Then Sect.~\ref{se:relext} provides a parameter space for relative extensions of $\cB_{S_{A_0}\times S_{B_0}}$ by $\cA_{S_{A_0}\times S_{B_0}}$. Denote the associated linear space in the sense of \eqref{eq:Exts} by $V_1 \to  {S_{A_0}\times S_{B_0}}$.
On the complement of the support of $\cH^1(\cP^\tbullet)$ only the trivial extension exists. The universal object is denoted by $0\to\cA_{V_1}\to \cF \to \cB_{V_1} \to 0$, where the associated families of subbundles and quotient bundles resp.\ are pulled back to $V_1$. This relative extension parameterizes all extensions of bundles $\cB_r$, $r\in S_{B_0}$ by bundles $\cA_s$,  $s\in S_{A_0}$. Let $\cA_{V_1} \hookrightarrow \cF$ be the associated family of short filtrations.

It will be needed that the  complex $\cP^\tbullet$ is defined over all of $S_{A_0}\times S_{B_0}$. This condition can be achieved beforehand: There exist only finitely many choices of isomorphism classes of $A_0$ and $B_0$ so that the space $U$ from Lemma~\ref{le:Psi} can be chosen to be valid for all such subbundles. Now $U$ also controls the parameter spaces $S_{A_0}$ and $S_{B_0}$. This fact implies that without loss of generality $\cP^\tbullet$ and hence its cohomology is defined on all of $S_{A_0}\times S_{B_0}$.

The crucial condition is related to the extension induced by $A \subsetneq E$ occurring as a fiber. This is guaranteed by the construction. If in a neighborhood all deformations of this trivial extension are also trivial, there is no need to act. If not, it is contained in the fiber of an irreducible component of $S_1={\rm supp }\, \cH^1(\cP^\tbullet)$. The further condition on $U$ is that all irreducible components of $S_1$ contain the origin. Also this condition can be satisfied beforehand for all finitely many previous choices. These arguments prove the following proposition.
\begin{proposition}\label{pr:sfilt}
  There exists a holomorphic family of short filtrations
  $$
  0\subsetneq \wt\cA \subsetneq \wt\cE
  $$
  over a linear fiber space $V\to R$, where $R$ is a closed subspace of $S_{A_0}\times S_{B_0}$ such that the fiber of the origin is
  $$
  0\subsetneq A_0 \subsetneq E_0
  $$
  and such that
  $$
  0 \subsetneq A \subsetneq E.
  $$
  corresponds to a point of $V$.
\end{proposition}
Note that $V$ also parameterizes a deformation of $E_0$ so that the induced deformation $\wt\cE$ of $E_0$ over $V$ is a pull-back of the semi-universal family.

\subsection{General families of filtrations}
The following consequence of Proposition~\ref{pr:homext} will be needed.
\begin{proposition}\label{pr:conn}
  Let $\cF$ and $\cG$ be holomorphic families of vector bundles over $X\times R$ with projection $pr:X\times R \to R $. Let $V_0 \to R$ be the linear space associated to $\cH om(pr; \cF,\cG)$ in the sense of Proposition~\ref{pr:homext}(i) with universal homomorphism $\Phi:\cF_{V_0} \to \cG_{V_0}$. Then there exists the (possibly empty) complement $R'\subset R$ of an analytic set and a subspace $V'_0 \subset V_0| R' \to R'$ parameterizing all isomorphisms $\Phi|\cF_r:\cF_r \to \cG_r$ existing for $r\in R'$.
\end{proposition}
\begin{proof}
  Unless the fibers of $\cF$ and $\cG$ are of the same dimension there is nothing to prove. Given $\Phi$ as above, there is a section $\det(\Phi)$ of a line bundle $ \det(\cF)^*\otimes \det(\cG)$ with zero set $N\subset X\times S$. The complement of $pr(N)$ defines $R'$, and $V'_0$ is defined in the obvious way.
\end{proof}

\label{se:genfil}
Again a polystable bundle $E_0$ is given together with a semi-universal deformation $\cE \to S \subset U$ following the construction summarized in Section~\ref{se:prl}.
Further choices will be made about the set $U\subset H^{0,1}(X,End(E_0))$ (see Sec.~\ref{se:cho}).

\begin{proposition}\label{pr:defilt}
There exists a semi-universal deformation of $E_0$ given by a family $\cE$ of semi-stable bundles on $X\times S$ with the following property.

Let $s\in S$ be a point such that $E=\cE_s$ is polystable and such that the corresponding element $\alpha\in H^{0,1}(X,End(E_0))$ satisfies $m(\alpha)=0$.  Then for any filtration of
\begin{equation}\label{eq:Es}
   0 \subsetneq E^1 \subsetneq \cdots\subsetneq E^k = E
\end{equation}
there exists a  space $W$ and a filtration
\begin{equation}\label{eq:EW}
  0 \subsetneq \cE^1 \subsetneq \cdots\subsetneq \cE^k
\end{equation}
over $X \times W$ defining a deformation of a certain induced filtration
\begin{equation}\label{eq:filtE0}
   0 \subsetneq E^1_0 \subsetneq \cdots\subsetneq E^k_0 = E_0
\end{equation}
of the polystable bundle $E_0$ together  with a map $W\to S$ such that $\cE^k$ is the pull-back of $\cE$ and \eqref{eq:Es} corresponds to a point $w\in W$ that is mapped to $s\in S$.
\end{proposition}
Note that the construction depends upon the initial choice of a point $s\in S$ such that $\cE_s$ is polystable.
\begin{proof}
First the construction will be given. Given the bundle $E$ as in \eqref{eq:Es} an associated filtration \eqref{eq:filtE0} is constructed by induction over $k$.

For $k=1$ there is nothing to show. However, it will be needed later that the components of the coarse moduli space of stable bundles are complements of analytic sets in the respective components of the analytic GIT-space. This was shown in \cite[Remark~5]{b-s}. For the convenience of the reader the short argument is indicated. Let the component $S^{(\nu)}\subset S$ contain stable points, which are known to form an open subset in the classical topology. Denote by $A\subset S^{(\nu)}$ (in the notation of \cite{b-s}) the set of points where the automorphism group of the fiber is of dimension greater than one. Then $A$ is invariant under $G$ in the restricted sense, and $\wt A\san2 G \subset \wt S^{(\nu)}\san2 G$ is a closed nowhere dense analytic subset (cf.\ \cite[Thm.~2]{b-s}). For any non-stable point of $\wt S^{(\nu)}$ the closure of the $G$-orbit contains a polystable but not stable point, whose image in the analytic GIT-quotient is in  $\wt A\san2 G$. Taking the preimage of the GIT-quotient of $\wt A$, it can be seen that the locus of non-stable points in $S^{(\nu)}$ is analytic.

The case $k=2$ refers to short filtrations. Both given filtrations will be split at the degree $k-1$ into a filtration of length $k-1$ and a short filtration. The hypothesis on the moment map for $E^{k-1}$ is satisfied by Lemma~\ref{le:Psi}(iv). The results about short filtrations can be applied. There exists
$$
{{\wt\cE}}^{k-1} \subsetneq {{\wt\cE}}^k
$$
with a parameter space which parameterizes the families of short filtrations and is now called $R^k$. Also there is a map $R^k \to S$ such that $\wt \cE^k$ is the pull-back of the semi-universal family $\cE\to S$ with respect to the base change map $R^k\to S$. There exists a point $r\in R^k$, $r\mapsto s$  such that the filtration
$$
\cE^{k-1}_r \subsetneq \cE^k_r
$$
is isomorphic to
$$
E^{k-1} \subsetneq E^k
$$
and over the origin the filtration is
$$
E^{k-1}_0\subsetneq E^k_0.
$$
Now the induction hypothesis can be applied to the truncated filtration
$$
0 \subsetneq E^1 \subsetneq E^2 \subsetneq \cdots\subsetneq E^{k-1}.
$$
Note that a filtration
$$
0 \subsetneq \wt{\wt\cE}^1 \subsetneq \cdots\subsetneq \wt{\wt\cE}^{k-1}
$$
exists with a parameter space $R^{k-1}$. Again the filtrations \eqref{eq:Es} and \eqref{eq:filtE_0} occur as fibers of the family taken at $t\in R^{k-1}$ after  truncation at $k-1$.

Both filtrations can pulled back to $R^{k-1}\times R^k$. Then both families are  glued together over a space $W\to  R^{k-1}\times R^k$ that classifies relative isomorphisms in the sense of Proposition~\ref{pr:homext}(i). Now Proposition~\ref{pr:conn} is applied: namely for the two degree $k-1$ contributions the fibers at the point $(t,r)$ are isomorphic so that these need to occur as images of $W \to R^{k-1}\times R^k$. Concerning the ``size'' of the parameter spaces, refer to Section~\ref{se:cho}.
\end{proof}

The construction of Proposition~\ref{pr:defilt} can be read as the construction of a complete deformation of a certain filtration of $E_0$ such that a given filtration of $\cE_s$ occurs as a fiber.

It also contains the construction of a complete deformation of any filtration of $E_0$ (always slopes being kept fixed). The existence of a semi-universal deformation also follows (the only ambiguity arises from the deformation of single bundles -- the relative ext-sheaves give rise to representable functors). Given complete deformations ``formal deformation theory'' also yields the existence of a semi-universal deformation.

The proposition contains the fact that for any $s\in S$ with polystable fiber  a suitable filtration of the bundle $\cE_s$ is a fiber of a family of filtrations. Little can be said off-hand about the other parameters. Note that the length of a \jh\ filtration is in general only semi-continuous (cf.\ Lemma~\ref{le:semicont} below).

\begin{corollary}\label{co:irrcomp}
  Let $S^{(\nu)} \subset S$ be an irreducible component (at $0\in S$). Then there exists a reduced space  $R$ over $S$ with a {\em surjective} holomorphic map $R\to  S^{(\nu)}$ and a filtration of $\cE_R$ (with locally free subsequent quotients) such that for a certain open subset all fibers are \jh\ filtrations.
\end{corollary}
Splitting filtrations will be needed.
\begin{proposition}\label{pr:polfib}
  Let a holomorphic family of filtrations
  $$
  \cE^1 \subsetneq \cdots \subsetneq \cE^k
  $$
  over $X\times \ul S$ (with locally free subsequent quotients) be given. Let the bundle $\cE_0$ over $s=0\in \ul S$ be polystable. Then there exists a (universal) parameter space $T\to \ul S$ such that the pull-back of the family of filtrations splits.
\end{proposition}
\begin{proof}
  The construction of Proposition~\ref{pr:homext} is applied to homomorphisms $\varphi_\ell : \cE^{\ell+1} \to \cE^\ell$. These are parameterized by a linear space over a closed analytic subspace of $S$. The condition $\varphi_k|\cE^\ell= id$ determines a closed analytic subset on the space of parameters.
\end{proof}

\subsection{Facts from local complex geometry}
Two basic classical facts about holomorphic maps $f: X \to Y$ of (pure dimensional) analytic sets and applications will be needed. Denote by $X_y$ the fiber $f^{-1}(y)$ of $y\in Y$. The topological rank $\rho(x)$ at a point $x\in X$ of such a map is defined by
$$
\rho(x) = \dim_x X - \dim_x X_{f(x)}.
$$
Always $\rho(x) \leq \dim Y$ holds, and openness of $f$ is equivalent to the topological rank being equal to $\dim Y$ everywhere.
{\it
\begin{itemize}
\item[(a)]
Provided $\rho(x)$ is constant, any $x\in X$ possesses a fundamental system $\{U\}$ of neighborhoods such that $f(U)\subset Y$ is analytic at $x$.
\item[(b)]
Given any fixed number $k$ the sets of points $X_k$ with $\rho(x)\leq k$ is analytic in $X$.
\end{itemize}
}
\begin{lemma}\label{le:locan}
  Let $f:X \to Y$ be a holomorphic map of reduced complex spaces with $X$ pure dimensional and $Y$ irreducible and locally irreducible.  Then for all $k$ there is a (locally) finite union
  $$
  f(X_k) = \bigcup_\nu Q^k_\nu
  $$
  where $ Q^k_\nu \subset W^k_\nu \subset Y$ are analytic subsets of (classically) open subsets  $W^k_\nu$ of $Y$. Such a set will be called {\em analytically constructible}.
\end{lemma}
\begin{proof}
  Apply (a) to the union of $f(X_\ell\backslash X_{\ell-1})$.
\end{proof}

\section{Existence of the coarse moduli space}\label{se:proofmain}
In Sec.~\ref{se:wenorm} the transition to weakly normal spaces was discussed. Weak normalizations of reduced spaces will be taken, if necessary, i.e.\ all occurring analytic sets will be equipped with their weakly normal complex structure -- a process which does not change the underlying set and its topology. Technically speaking for holomorphic maps there exists a lift to the weak normalizations. The main theorem is the following:

\begin{theorem}\label{th:main}
Given a compact \ka\ manifold the classifying space $\wh\cM_{GIT}$ is a coarse moduli space for polystable vector bundles with distinguished topology in the category of weakly normal spaces.
\end{theorem}

\subsection{Outline of the proof}
Conditions (i) and (ii) of Definition~\ref{de:coarse} follow from the main theorem of \cite{b-s}. However, condition (iii) of the definition, namely the uniqueness of a coarse moduli is critical, because of the absence of a deformation  theory for polystable bundles. If $\cQ$ denotes another candidate, the issue is the holomorphicity of the natural map $\wt S\san2 G \to \cQ$. It is sufficient to consider an irreducible component $S^{(\nu)} \subset S$. In general such a component does not consist of polystable points so that even pointwise a map from $S^{(\nu)} \to \cQ$ is defined primarily via the analytic GIT-quotient $\wt S^{(\nu)}\san2 G$. As the length of a \jh\ filtration in general is not constant in a holomorphic family, filtrations are introduced that only have the \jh\ property generically. Degenerating parameters, where \jh\ filtrations are longer, cannot be avoided. In the next step such filtrations with a splitting are considered. It is shown that finitely many of those with parameter spaces $T$ are sufficient so that the images of the maps   $T\to S^{(\nu)}$ followed by the projection to $\wt S^{(\nu)}\san2 G$ cover the whole local analytic GIT-quotient. It is not clear how to see directly that a map to $\cQ$ exists, but on the complement of a thin subset there exist polystable fibers so that a pointwise definition of a map to $\cQ$ that is defined in terms of isomorphism classes of polystable points yields holomorphicity, and finally a holomorphic extension to all of the parameter spaces $T$ with values in $\cQ$ exists, which is sufficient for $\wt S\san2 G \to \cQ$ to be holomorphic.

\subsection{\jh\ filtrations}
In this section the results from \cite{bu-sch} cited in Proposition~\ref{pr:destable} will be crucial. A main property of the situation is that the lengths of \jh\ filtrations (being independent of the choice of such a filtration) need not be constant -- a fact that requires special attention.  Later the loci of constant length \jh\ filtrations will have to be considered separately.

\begin{lemma}\label{le:semicont}
  The length of a \jh\ filtration for a family $\cE$ over $S$ is upper semi-continuous: The set of parameters with length less than or equal to a certain number $k$ is open.
\end{lemma}
The lemma implies that in the situation of Corollary~\ref{co:irrcomp} the filtration can be taken as a \jh\ filtration (of minimal length within the family) over an open subset of $S^{(\nu)}$.
\begin{proof}[Proof of the Lemma]
The open set $U$ containing the parameter set $S$ was chosen as a ball (see Sec.~\ref{se:prl} and \cite[Sec.~1]{bu-sch}) such that the closure of any $G$-orbit taken in $S$ contains a polystable point (cf.\ the construction of an analytic GIT-quotient in \cite[Sec.~2]{b-s}, in particular also the Kempf-Ness theorem). As the length of a \jh\ filtration of a fiber is the same for all points in the closure of a $G$-orbit, it is sufficient to consider polystable points. Now Corollary~\ref{co:numsum} implies the claim.
\end{proof}
Recall (Prop.~\ref{pr:destable}) that all bundles $E$ close to $E_0$ are semi-stable and all subsheaves of $E$ of the same slope and with torsion-free quotient are subbundles,
a fact that can be applied to filtrations.
\begin{lemma}\label{le:sesh}
Let a filtration as above
  \begin{equation}\label{eq:sfi}
  E^1\subsetneq E^2 \subsetneq\cdots \subsetneq E^k
  \end{equation}
of locally free sheaves on $X$ be given such that all subsequent quotients are locally free with the same slope. Suppose that $k$ is the length of a \jh\ filtration of $E^k$.  Then \eqref{eq:sfi} is a \jh\ filtration.
\end{lemma}
\begin{proof}
The filtration \eqref{eq:sfi} can be refined so that the result is a \jh\ filtration with locally free graded objects using \cite[Prop.~4.4]{b-s}. Since the length of a \jh\ filtration is independent of the choice, the claim follows.
\end{proof}
Given $S\san2 G\subset \cM_{GIT}$, denote by $S^{(\nu)}\subset S$ the local irreducible components taken at $s_0=0$.

\subsection{Application to analytic GIT-quotients}
A component $S^{(\nu)}$ (or rather its extension $\wt S^{(\nu)}$ in the sense of \cite[Thm.~1]{b-s}) obviously gives rise to a local irreducible component of the analytic GIT-quotient.

Given any polystable point $s$ of $S^{(\nu)}$ contained in the zero set of the moment map, a holomorphic family of splitting filtrations parameterized by a space $T\to S^{(\nu)}$ was constructed in Proposition~\ref{pr:polfib}, which yielded a \jh\ filtration at the point $s$ and a filtration of $E_0$ with not necessarily stable subsequent quotients. The locus where not all of the respective quotients are stable can be identified by the dimension of the automorphism groups. Hence it is an analytic subset of $T$, which is nowhere dense, because stability is an open condition. Denote the complement by $T'\subset T$.

The Kempf-Ness theorem yields that $\wt S^{(\nu)}\san2 G$ is the image of the set of {\em polystable} points with vanishing moment map.

It follows from Lemma~\ref{le:iso} that only finitely many such types exist up to the action of $\Gamma$ resp. $G$. Hence only finitely many image spaces of spaces of the type $T$ are relevant in the GIT-quotient.

\begin{proposition}
The space $\wt S^{(\nu)}\san2 G$  is covered by finitely many images of spaces $T\to S^{(\nu)}$ which parameterize families of splitting filtrations. On the complements $T'\subset T$ of thin analytic subsets the corresponding filtrations consist of splitting filtrations of the \jh\ type.
\end{proposition}

\begin{proof}[Proof of Theorem~\ref{th:main}]
The classifying space $\wh{\cM}_{GIT}$ for polystable vector bundles satisfies conditions (i) and (ii) from Definition~\ref{de:coarse} by the main theorem of \cite{b-s}. In order to prove (iii) pick another candidate $\cQ$ for a coarse moduli space. By (i) there exists a unique set-theoretic map $\chi: \wh{\cM}_{GIT}\to \cQ$, which is continuous by the additional assumption on the topology. The aim is to show that $\chi$ restricted to the  complement of a certain thin analytic set in $\wh{\cM}_{GIT}$ is holomorphic so that the Riemann extension theorem on the weakly normal space $\wh{\cM}_{GIT}$ yields holomorphicity of $\chi$.

In the sequel the statement of  Lemma~\ref{le:hol} below will be tacitly used. The holomorphic map $S \to \wt S\san2 G$ is surjective and the image of $S^{(\nu)}$ is an irreducible component. There is the following situation:
$$
\xymatrix{
T\ar[r]  \ar[drr]_{\wh \chi} & \wt S^{(\nu)}\san2 G \ar[dr]^{\wt\chi}\ar@{^{(}->}[r] & \wh \cM_{GIT}\ar[d]^\chi\\
&& \cQ
}
$$
It will be sufficient to show that for all $T$ all maps $\wt\chi$ are {\em holomorphic} so that a holomorphic extension to $\wt S^{(\nu)}\san2 G$ exists.

Note that the map $\wh\chi: T \to \cQ$ so far is only defined in terms of the analytic GIT-quotient. Over its restriction to $T'$ the fibers are \jh\ filtrations of polystable bundles. The restriction of $\wh\chi$ to $T'$ sends a point to the isomorphism class in $\cQ$ of the polystable bundle which is at the end of the filtration. As such this map from $T'$ to $\cQ$ is holomorphic. Because of the assumptions on the topology of a coarse moduli space and the weakly normal structure of all occurring reduced complex spaces, the map $\wt \chi$ is holomorphic on $T$.

Finally the holomorphicity of the map $\wt\chi$ has to be shown. The image of $T \to \wt S^{(\nu)}\san2 G$ is analytically constructible. As only finitely many choices of families $T$ are relevant, the image of all possible spaces $T$ has to contain classically open sets (see Lemma~\ref{le:locan}). On the union of those $\wt\chi$ is holomorphic. On the other hand the union of these is mapped onto all of the local analytic GIT-quotient. Hence locally $\wt\chi$ is known to be holomorphic on the complement of thin analytic subsets. Hence $\wt\chi$ is holomorphic on all of $\wt S^{(\nu)}\san2 G$.
\end{proof}

The following fact is stated in terms of notation that is not related to that in the rest of the paper.

\begin{lemma}\label{le:hol}
  Let $\varphi:X\to Y$, and $\psi:X\to Z$ be holomorphic maps of weakly normal complex spaces, and $\chi: Y \to Z$ be a set-theoretic map such that $\psi = \chi\circ\varphi$. Suppose that $\varphi$ is an open, surjective map. Then $\chi$ is holomorphic.
\end{lemma}

\begin{proof}
Let $W\subset Z$ be an open subset. As $\varphi$ is surjective $\chi^{-1}(W)=\varphi(\psi^{-1}(W))\subset Y$ is open and $\chi$ is continuous.   Denote by $\wt \psi: \cO_Z(W) \to \cO_X(\psi^{-1}(W))$ the map defined by $h\mapsto g\circ \psi$ and $\wt \varphi : \cO_Y(V) \to \cO_X(\varphi^{-1}(V))$ the map given by $f \mapsto f\circ \varphi$ for open sets $W\subset Z$ and $V\subset Y$. The analogous map $\wt\chi$ takes $\cO_Z(W)$ to the set of continuous $\C$-valued functions $\mathrm{C}^0_Y(\chi^{-1}(W))$ on $\chi^{-1}(W)$. The map $\wt\varphi$ extends to a map $\mathrm{C}^0_Y(\chi^{-1}(W))\to \mathrm{C}^0_X(\psi^{-1}(W))$ which is injective because of the surjectivity of $\varphi$. The map $\varphi : X \to Y$ is restricted to regular loci, namely to  $\varphi^{-1}(Y_{reg})\cap X_{reg}$, and further to the locus of maximal rank. On the complement $Y'\subset Y$ of these  thin analytic sets there exist local analytic inverse maps $\eta$ in the sense that $\varphi\circ\eta=id_U$ for $U \subset Y'$ chosen suitably. The claim is that the inverse image of any holomorphic function on $\psi^{-1}(W)$ under the map $\wt\varphi: \mathrm{C}^0_Y(\chi^{-1}(W))\to \mathrm{C}^0_X(\psi^{-1}(W))$ is holomorphic: Let $f$ be continuous on $\psi^{-1}(W)\subset X$ such that $f\circ \varphi$ is holomorphic. Then on such an open set $U \subset \chi^{-1}(W)$ the function $ f\circ \varphi\circ \eta = f|U$ is again holomorphic. Hence $f$ is holomorphic outside a thin analytic subset, and by continuity and weak normality $f$ is holomorphic on all of $\chi^{-1}(W)$.
\end{proof}

\begin{corollary}
  Suppose that (for a given connected component) $\cM_{GIT}$ of the analytic GIT-space all deformations of polystable bundles are unobstructed. Then the normal space $\cM_{GIT}$ is a coarse moduli space with distinguished topology of polystable bundles.
\end{corollary}

\end{document}